\documentclass[10pt]{article}
\usepackage{graphicx}
\usepackage{amsfonts, amsmath, amsthm, amssymb, latexsym, amscd}
\newcommand{\comment}[1]{}
\newtheorem{thm}{Theorem}[section]

\newtheorem{lem}[thm]{Lemma}

\newtheorem{obs}[thm]{Observation}
\newtheorem{prop}[thm]{Proposition}

\newtheorem{cor}[thm]{Corollary}

\begin{document}

\title{Total dominator coloring of\\
circulant graphs $C_n(a,b)$}
\author{Adel P. Kazemi$^{1}$\footnote{Corresponding author}, Parvin Jalilolghadr
and Abdollah Khodkar$^2$ \\[1em]
$^{1}$ Department of Mathematics, \\ University of Mohaghegh Ardabili, \\ P.O.\ Box 5619911367, Ardabil, Iran. \\[1em]
$^3$       Department of Mathematics, \\ University of West Georgia, \\ Carrollton, GA 30118, USA.\\[1em]
$^1$ Email: adelpkazemi@yahoo.com\\
$^2$ Email: akhodkar@westga.edu}
\date{}
\maketitle
\begin{abstract}
The circulant graph $C_n(S)$ with connection set $S\subseteq \{1,2,\cdots,n\}$ is the graph
with vertex set $V=\{1,\ldots, n\}$ and two vertices $x,y$ are adjacent
if $|x-y|\in S$. In this paper, we will calculate the total dominator chromatic number
of the circulant graph $C_n(\{a,b\})$ when $n\geq 6$, $gcd(a,n)=1$ and $ a^{-1}b\equiv  3 \pmod{n}$.
\end{abstract}

\textbf{Keywords:} Circulant graph, Total dominator chromatic number, Total domination number.
\newline \textbf{AMS subject classification:} 05C15; 05C69.
\newline Received: Fev. 18, 2019, Accepted: April 23, 2019.
\newline\indent{\scriptsize $\copyright$ 2019 Utilitas Mathematica}
\pagestyle{myheadings}
\markboth{\rightline {\scriptsize  A. P. Kazemi}}
         {\leftline{\scriptsize A. P. Kazemi and et al. ----------------------- Total dominator coloring of circulant graphs $C_n(a,b)$}}                         

\section{\bf Introduction}

All graphs considered here are finite, undirected and simple. For
standard graph theory terminology not given here we refer to \cite{West}. Let $%
G=(V,E) $ be a graph with the \emph{vertex set} $V$ of \emph{order}
$n(G)$ and the \emph{edge set} $E$ of \emph{size} $m(G)$. The
\emph{open neighborhood} and the \emph{closed neighborhood} of a
vertex $v\in V$ are $N_{G}(v)=\{u\in V\ |\ uv\in E\}$ and
$N_{G}[v]=N_{G}(v)\cup \{v\}$, respectively. The \emph{degree} of a
vertex $v$ is also $deg_G(v)=|N_{G}(v)|$.
If $\deg(v)=k$ for every vertex $v\in V$, then $G$ is called $k$-\emph{regular}.
For a subset $S$ of $V$, the set $N(S)=\bigcup_{s\in S} N(s)$ is called the open neighborhood of $S$.

A subset $L$ of vertices of a graph $G$ is an \emph{open packing} if the
open neighborhoods of vertices in $L$ are pairwise disjoint. The  \emph{open packing number} $\rho^o(G)$ of $G$ is the maximum cardinality of an open packing in $G$.

Let $1\leq a_1<a_2<\cdots<a_m\leq \lfloor\frac{n}{2}\rfloor$, where $m,n,a_i$ are integers, $1\leq i\leq m$, and $n\geq 3$. Set $S=\{a_1, a_2,\cdots, a_m\}$. A graph $G$ with the vertex set $\{1, 2, \cdots, n\}$ and the edge set $$
\{\{i,j\}\mid |i-j|\equiv a_t \pmod{n} \mbox{ for some } 1\leq t\leq m\}
$$ is called  a \emph{circulant graph} \cite{MNR} with respect to set $S$ (or with connection set $S$), and denotes by $C_n(S)$ or $C_n(a_1, a_2,\cdots, a_m)$. Notice that $C_n(S)$ is $k$-regular, where $k=2|S|-1$ if $\frac{n}{2}\in S$ and $k=2|S|$ otherwise.

Circulant graphs have a vast number of uses and applications to telecommunication network, VLSI design, parallel and distributed computing. 
The study of chromatic number of circulant graphs of small degree has been widely considered in some literatures (for example see \cite{Heu,MNR,YZ}). Hamiltonicity properties of circulant graphs have been studied by Burkard and Sandholzer \cite{Reb}, who proved that a circulant graph is Hamiltonian if and only if it is connected. Also, in \cite{Heu}, the planarity and bipartiteness of circulant graphs have been argued.

A \emph{total dominating set}, briefly TDS, $S$ of a graph $G$ \cite{HeYe13} is a subset
of the vertices in $G$ such that for each vertex $v$, $N_G(v)\cap
S\neq \emptyset$. The \emph{total domination number $\gamma_t(G)$} of $G$ is the minimum cardinality of a TDS of $G$. A vertex $v$ in $G$ \emph{totally dominates} a subset $X$ of vertices in $G$, written $v \succ_t X$, if $X\subseteq N(v)$, that is, if $v$ is adjacent to every vertex in $X$. Also $v \not\succ_t  X$ means that the vertex $v$ does not totally dominate $X$.

A \emph{proper coloring} of a graph $G =(V,E)$ \cite{West} is a function from
the vertices of the graph to a set of colors such that any two
adjacent vertices receive different colors. The \emph{chromatic number}
$\chi (G)$ of $G$ is the minimum number of colors needed in a proper
coloring of a graph. In a proper coloring of a graph, a \emph{color class} of the
coloring is a set consisting of all those vertices with the same
color. If $f$ is a proper coloring of $G$ with the coloring classes
$V_1,~V_2,~ \ldots,~V_{\ell}$ such that every vertex in $V_i$ has
color $i$, we write simply $f=(V_1,V_2,\ldots,V_{\ell})$.

Motivated by the relation between coloring and domination, the notion of total dominator coloring was introduced in \cite{Kaz2015}. Also the reader can consult \cite{Hen2015,Kaz2014,Kaz2016,KazKaz2018}  for more information.
A \emph {total dominator coloring}, briefly TDC, of a graph $G$ is a proper coloring of $G$ in
which each vertex of the graph is adjacent to every vertex of some
color class. The \emph {total dominator chromatic number} $\chi_d^t(G)$
of $G$ is the minimum number of color classes in a TDC of $G$.

Let $f=(V_1,V_2,\ldots,V_{\ell})$ be a TDC of $G$.
A vertex $v$ is called a \emph{common neighbor} of
$V_i$ if $v\succ_t V_i$, \cite{Kaz2015}. The set of all common neighbors of $V_i$ is
called the \emph {common neighborhood} of $V_i$ in $G$ and denoted by
$CN_G(V_i)$ or simply by $CN(V_i)$.

It is proved in \cite{Hen2015} that for any graph $G$ without an isolated vertex,
\begin{eqnarray}\label{max{chi,gamma_t} =< chi_d^t}
\max\{\chi(G),\gamma_t(G)\} \leq \chi_d^t(G)\leq \gamma_t(G)+\chi(G).
\end{eqnarray}
Also obviousely, every proper coloring $f=(V_1,\ldots,V_{\ell})$ of a graph $G$ of order $n$ is a TDC of $G$ if and only if
\begin{eqnarray}\label{cup CN(V_i)=V(G)}
\bigcup\limits_{i=1}^{\ell} CN_G(V_i)=V(G).
\end{eqnarray}
So if $f=(V_1,\ldots,V_{\ell})$ is a TDC, then
\begin{eqnarray}\label{sum v_i'>=n}
\sum\limits_{i=1}^{\ell} |CN_G(V_i)|\geq n.
\end{eqnarray}

In this paper, for a proper coloring $f=(V_1,V_2,\ldots,V_{\ell})$ of a circulant graph $G$, we assume $v_i=|V_i|$, $v_i'=|CN_G(V_i)|$, and $v_1\geq v_2\geq \cdots \geq v_{\ell}$ (so $v_1$ is at most $\alpha (G)$, the independence number of $G$).

For integers $a_1,a_2,\cdots,a_{\ell}\geq 0$, by $(v_1,v_2, \cdots ,v_{\ell}) \leq (a_1,a_2, \cdots ,a_{\ell})$ we mean that $v_i \leq a_i$ for each $i$. When vertices of $G$ are labeled 1, 2, $\cdots$, $n$, then we may use $V_o$ (and $V_e$)
to denote vertices with odd (even) labels.

The main goal of this paper is to prove the following theorem.

\begin{thm} \label{TDCN Cn(a,b)}
Given the circulant graph $C_n(a,b)$, where $n\geq 6$, $gcd(a,n)=1$ and $ a^{-1}b\equiv  3 \pmod{n}$, we have
\begin{eqnarray*}
\chi_{d}^t(C_n(a,b))=\left\{
\begin{array}{ll}
2\lceil \frac{n}{8} \rceil & \mbox { if } 8\leq n \leq 10,\\
2\lceil \frac{n}{8}\rceil +1     & \mbox{ if }n\equiv 1\pmod{8} \mbox{ or } n=11,\\
2\lceil \frac{n}{8}\rceil +2    & \mbox{ otherwise.}
\end{array}
\right.
\end{eqnarray*}
\end{thm}

The proof of the following theorem can be found in \cite{Heu}.
\begin{thm} \label{HeuTheorem}
{\rm
Let $C_n(a,b)$ be a circulant graph and $gcd(a,n)=1$. Then the graph $C_n(a,b)$ is isomorphic to the graph $C_n(1, c)$ where $c \equiv a^{-1}b \pmod{n}$.
}
\end{thm}

Hence by Theorem \ref{HeuTheorem}, in order to prove Theorem \ref{TDCN Cn(a,b)}, it is sufficient to prove
the following result.

\begin{thm} \label{chi_{d}^t(C_n(1,3))}
{\rm
For any $n\geq 6$,
\begin{eqnarray*}
\chi_{d}^t(C_n(1,3))=\left\{
\begin{array}{ll}
2\lceil \frac{n}{8} \rceil & \mbox{ if } n=6 \mbox{  or } 8\leq n \leq 10,\\
2\lceil \frac{n}{8}\rceil +1     & \mbox{ if }n\equiv 1\pmod{8} \mbox{  or } n=11,\\
2\lceil \frac{n}{8}\rceil +2    & \mbox{ otherwise.}
\end{array}
\right.
\end{eqnarray*}
}
\end{thm}

\section{Preliminary Results}

In this section, we will calculate the independence and the packing numbers of the circulant graphs $C_n(1,3)$ and also we will present some crucial lemmas. First an observation.
Recall that for a proper coloring $f=(V_1,V_2,\ldots,V_{\ell})$ of a circulant graph $G$, we assume $v_i=|V_i|$, $v_i'=|CN_G(V_i)|$, and $v_1\geq v_2\geq \cdots \geq v_{\ell}$. We also assume the vertex set of the circulant graph $C_n(1,3)$ is $\{1,2,3,\ldots,n\}$.

\begin{obs} \label{Obs.1}
{\rm
For any proper coloring of the circulant graph $C_n(1,3)$ of order $n\geq 9$,
\begin{itemize}
\item $v_i+v_i'\leq 5$ if $v_i \leq 4$, and 
\item $v_i'=0$ if $v_i \geq 5$.
\end{itemize}
}
\end{obs}

\begin{prop} \label{alpha_Cn(1,3)}
{\rm
For any $n\geq 4$,
\begin{eqnarray*}
\alpha(C_n(1,3))= \left\{
\begin{array}{ll}
\frac{n}{2}            & \mbox{ if }n \mbox{ is even,} \\
\frac{n-3}{2} & \mbox{ if }n \mbox{ is odd.}
\end{array}
\right.
\end{eqnarray*}
}
\end{prop}

\textbf{Proof: }
Let $L$ be an independent set in $C_n(1,3)$. Then the distance between the vertices in L must be at least 2. Hence, if $n$ is even, then the even vertices form an independent set of the largest size. If $n$ is odd, then there is an edge between vertex $n-1$ and $2$. Hence, the even vertices minus vertex $n-1$ form an independent set of the largest size. This completes the proof.
$\square$ 


\begin{prop} \label{open-packing Cn(1,3)}
{\rm
For any $n\geq 3$,
\begin{eqnarray*}
\rho^o(C_n(1,3))= \left\{
\begin{array}{ll}
\lfloor\frac{n}{3}\rfloor & \mbox{ if } 3\leq n\leq 6,\\
\lfloor\frac{n}{4}\rfloor-1 & \mbox{ if } n\equiv 4,6\pmod{8},\\
\lfloor\frac{n}{4}\rfloor  & \mbox{ otherwise.}
\end{array}
\right.
\end{eqnarray*}
}
\end{prop}

\textbf{Proof: }
If $3\leq n \leq 6$, then it is easy to see that
$\rho^o(C_n(1,3))=\lfloor \frac{n}{3}\rfloor$. Now let
$n\geq 7$. By the 4-regularity of
$C_n(1,3)$, we have $\rho^o(C_n(1,3))\leq \lfloor \frac{n}{4}\rfloor$.

Let $n\equiv 4,6 \pmod{8}$, and let $L$ be an open packing in $C_n(1,3)$ with cardinality
$\lfloor\frac{n}{4}\rfloor$. Let $V(C_n(1,3))=V_o\cup V_e$ be the
partition of the vertex set of the graph to the set of odd numbers
$V_o$ and the set of even numbers $V_e$. For simplicity set $n_e=|L\cap V_e|$ and $n_o=|L\cap V_o|$.
Since $n$ is even, it follows that $N(i)\subseteq V_o$ if $i\in L\cap V_e$
and $N(j)\subseteq V_e$ if $j\in L\cap V_o$. We may assume that $n_e>n_o$. Since $n_o+n_e=\lfloor\frac{n}{4}\rfloor=2k+1$, $n_e\geq k+1$. So the number of vertices dominated by $n_e$ even vertices is at least $4k+4$ which is a contradiction, because the number of odd vertices is at most 4k+3. \\
Therefore, for any $n\geq 7$,
\begin{eqnarray*}
\rho^o(C_n(1,3))\leq \left\{
\begin{array}{ll}
\lfloor\frac{n}{4}\rfloor-1  & \mbox{ if } n\equiv 4,6\pmod{8} , \\
\lfloor\frac{n}{4}\rfloor  & \mbox{ otherwise}.
\end{array}
\right.
\end{eqnarray*}
Now since the sets $\{1+8t, 2+8t~|~0\leq t\leq
\lfloor\frac{n}{8}\rfloor-1\}$, $\{8t, 8t-5~|~1\leq t\leq
\lfloor\frac{n}{8}\rfloor\}\cup\{n\}$ and $\{1+8t,
2+8t~|~0\leq t\leq \lfloor\frac{n}{8}\rfloor-1\}\cup\{n-6\}$ are all
open packing for $C_n(1,3)$ when $n\not\equiv 5,7\pmod{8}$, $n\equiv
5\pmod{8}$ and $n\equiv 7\pmod{8}$, respectively, the equality is obtained.
$\square$ 
\begin{obs}\label{obs2}
{\rm
Let $u,v,w\in V(C_n(1,3))$.
\begin{itemize}
\item If the distance between $u$ and $v$ on the cycle $(1,2,\ldots,n)$ is $2, 4, 6$, then
$|CN(\{u,v\})|=3,2,1$, respectively, otherwise $|CN(\{u,v\})|=0$.

\item If the distances between $u$ and $v$, $v$ and $w$, $u$ and $w$ are $2,2,4,$ respectively, then
$|CN(\{u,v,w\})|=2$ and if the distances are $2,4,6$, then $|CN(\{u,v,w\})|=1$, otherwise $|CN(\{u,v,w\})|=0$.
\end{itemize}
}
\end{obs}
\begin{obs}\label{obs3}
{\rm
Let $L$ be an open packing in $C_n(1,3)$ of cardinality $\rho^o(C_n(1,3))$.
\begin{itemize}
\item If $n\not\equiv 5, 7\pmod 8$, then the subgraph of $C_n(1,3)$ induced by the vertices in $L$
consists of $\lfloor\frac{n}{8}\rfloor $ edges.

\item If $n\equiv 5, 7\pmod 8$, then the subgraph of $C_n(1,3)$ induced by the vertices in $L$
consists of $\lfloor\frac{n}{8}\rfloor$ edges and an isolated vertex.
\end{itemize}
}
\end{obs}

The next lemma states that for even $n$ how large the cardinality of an independent set in $C_n(1,3)$ can be
when it contains numbers with different parity.

\begin{lem} \label{Max.ind.dif.parity}
{\rm
For even $n\geq 8$, let $I$ be an independent set in $C_n(1,3)$ such
that the numbers in $I$ have different parity. Then $|I|\leq \frac{n}{2}-3$.
}
\end{lem}

\textbf{Proof: }
Let $I$ be an arbitrary independent set in $C_n(1,3)$ with the
partition $I= A_1\cup \cdots \cup A_t$, where $t\geq 2$, such the set
$A_i$ consists of odd numbers if and only if $i$ is odd. Let also
\[
A_i=\{x_1^{(i)},\ldots, x_{a_i}^{(i)}\}
\]
for $1\leq i\leq t$, where $|A_i|=a_i$. Without loss of
generality, we may assume $x_1^{(i)}< x_2^{(i)}< \cdots <
x_{a_i}^{(i)}$ and $x_1^{(1)} =1$. Then we have
$x_1^{(i+1)}\geq x_{a_i}^{(i)}+5$ for $1\leq i\leq t-1$. Therefore
\[
n  \geq  \sum\limits_{i=1}^{t}(2a_i+3)
              =      2|I|+3t,
              \geq  2|I|+6,
\]
which implies that $|I|\leq \frac{n}{2}-3$.
$\square$ 

\begin{cor}\label{cor1}
{\rm
Let $f=(V_1,\cdots, V_{\ell})$ be a proper coloring of $C_n(1,3)$ for even $n \geq 18$. If $|V_i|\geq \frac{n}{2}-2$ for some $i$, then
$f$ is not a TDC.
}
\end{cor}

\textbf{Proof: }
By Lemma \ref{Max.ind.dif.parity}, the vertices in $V_i$ have the same parity. Without loss of generality, we may assume the vertices in $V_i$ are odd. Since $N(x) \subseteq V_e$ for any odd $x$, $N(y)\subseteq V_o$  for any even $y$ and $|V_o\cap(\bigcup_{j=1,j\neq i}^{\ell}V_j)|\leq 2$, the condition $n\geq 18$ implies that there exists
an even $x$ such that $x\nsucc_t V_j$ for each $1\leq j \leq \ell$, a contradiction.
$\square$ 
\section{A Lower Bound}
In this section we find a lower bound for the total dominator chromatic number of $C_n(1,3)$ for $n\geq 7$.
We make use of the following result in this section.

\begin{thm} \emph{\cite{Rad}}
\label{TDN C_n(1,3)}
{\rm
For any $n\geq 4$,
\begin{eqnarray*}
\gamma_t(C_n(1,3))= \left\{
\begin{array}{ll}
\lceil\frac{n}{4}\rceil+1 & \mbox{  if } n\equiv
2,4\pmod{8} , \\
\lceil\frac{n}{4}\rceil  & \mbox{ otherwise.}\\
\end{array}
\right.
\end{eqnarray*}
}
\end{thm}

First we find lower bounds for the total dominator chromatic number of $C_n(1,3)$ when $n\in\{7,8,9,10,11,12,14,18,19\}$.

\begin{lem} \label{TDCN the 7 numbers}
{\rm
For any circulant graph $C_n(1,3)$,
\begin{eqnarray*}
\chi_{d}^t(C_n(1,3)) \geq \left\{
\begin{array}{ll}
2\lceil \frac{n}{8} \rceil & \mbox{ if } 8\leq n \leq 10,\\
2\lceil \frac{n}{8}\rceil +1     & \mbox{ if } n=11,\\
2\lceil \frac{n}{8}\rceil +2    & \mbox{ if } n=7,12,14,18,19.
\end{array}
\right.
\end{eqnarray*}
}
\end{lem}

\textbf{Proof: }
We know
\begin{eqnarray*}
\begin{array}{ll}
\chi_{d}^t(C_n(1,3))\geq \chi(C_n(1,3))=2\lceil \frac{n}{8} \rceil +2 & \mbox{ if } n=7, \mbox{     ~~~~(by (\ref{max{chi,gamma_t} =< chi_d^t}))}\\
\chi_{d}^t(C_n(1,3))\geq \gamma_t(C_n(1,3))=2\lceil \frac{n}{8} \rceil & \mbox{ if } n=8,10, 
\mbox{ (by (\ref{max{chi,gamma_t} =< chi_d^t})+Theo. \ref{TDN C_n(1,3)})}.
\end{array}
\end{eqnarray*}
Let $f=(V_1,V_2,V_3)$ be a TDC of $C_9(1,3)$, $v_i=|V_i|$ and $v_i'=|CN_G(V_i)|$. Then
\[
\sum\limits_{i=1}^{3}v_i' \leq  \sum\limits_{i=1}^{3}(5-v_i)
                                        =  15-\sum\limits_{i=1}^{3}v_i
                                        =  6
                                        <  9,
\]
which contradicts (\ref{sum v_i'>=n}). Hence, $\chi_{d}^t(C_9(1,3)) \geq 4$. In a similar fashion, it can be proved that $\chi_d^t(C_{11}(1,3))\geq 5$. Now we prove $\chi_{d}^t(C_n(1,3)) \geq 2\lceil \frac{n}{8}\rceil +2$ for $n=12,14,19$. Since their proofs are similar, we only prove $\chi_d^t(C_{12}(1,3))\geq 6$.

Let $f=(V_1, \cdots, V_5)$ be a proper coloring of $C_{12}(1,3)$, where $v_1\geq v_2\geq \cdots \geq v_{5}$, and let $m=|\{i~|~v_i\geq 5\}|$. Note that $m=0$ or 1. First let $m=0$. Then $(v_1,\cdots, v_5)\in \{(4,3,2,2,1),(4,2,2,2,2),(3,3,3,2,1)\}$. Let $(v_1,\cdots, v_5)=(4,3,2,2,1)$ (other cases are similar). Hence, $(v_1',\cdots, v_5') \leq (1,2,3,3,$ $4)$ by Observation \ref{Obs.1}. If the numbers in some $V_i$ have different parity, then $|V_o|=|V_e|$=6 implies that the numbers in some $V_j$, $i\neq j$, have different parity. Hence, $v_i'=v_j'=0$ and so $\sum\limits_{i=1}^{5}v_i' < 12$, a contradiction with (\ref{sum v_i'>=n}).
Therefore, we assume the numbers in each $V_i$ have the same parity. Then $V_i$ is a subset of $V_e$ (or $V_o$) if and only if $CN_{C_{12}(1,3)}(V_i)$ is a subset of $V_o$ (or $V_e$). Without loss of generality, let $V_1\subseteq V_e$. Then either $V_3\subseteq V_e$ or $V_4\subseteq V_e$ (not both). Therefore
\[
\begin{array}{ll}
V_o=CN_{C_{12}(1,3)}(V_1) \cup CN_{C_{12}(1,3)}(V_3) \mbox{ or}\\
V_o=CN_{C_{12}(1,3)}(V_1) \cup CN_{C_{12}(1,3)}(V_4),
\end{array}
\]
which contradicts the fact that $|CN_{C_{12}(1,3)}(V_1) \cup CN_{C_{12}(1,3)}(V_i)| < |V_o|$ for $i=3,4$.

Now let $m=1$. Then $5\leq v_1\leq 6=\alpha (C_{12}(1,3))$. Lemma \ref{Max.ind.dif.parity} implies that the vertices in $V_1$ have the same parity. Without loss if generality, we may assume the numbers in $V_1$ are odd. Hence, $|(V_2\cup \cdots \cup V_5)\cap V_o|\leq 1$. If $v_1=6$, then $|(V_2\cup \cdots \cup V_5)\cap V_o|=0$, and so for any even $v$, $v\nsucc_t V_i$ for each $i$. Now assume $v_1=5$. Then $|(V_2\cup \cdots \cup V_5)\cap V_o|=1$. If $V_o=V_1\cup V_5$, then for some even $v$, we will have $v\nsucc_t V_i$ for each $i$. Now assume $V_o\neq V_1\cup V_5$. Without loss of generality, we may assume the numbers in $V_2$ have different parity, and so $v_2'=0$. Also we know $v_1=5$ implies that $(v_2,\cdots, v_5)=(4,1,1,1)$ or $(v_2,\cdots, v_5)=(3,2,1,1)$, or $(v_2,\cdots, v_5)=(2,2,2,1)$, and so $(v_2',\cdots, v_5')\leq (1,4,4,4)$ or $(v_2',\cdots, v_5')\leq (2,3,4,4)$ or $(v_2',\cdots, v_5')\leq (3,3,3,4)$, respectively. Let $(v_2',\cdots, v_5')=(1,4,4,4)$ (the proof of the other cases are left to the reader). Since $V_3\cup V_4\cup V_5 \subseteq V_e$, we obtain
\[
\bigcup\limits_{i=3}^{5} CN_{C_{12}(1,3)}(V_i)\subseteq V_o,
\]
and so $CN_{C_{12}(1,3)}(V_1)=CN_{C_{12}(1,3)}(V_2)=\emptyset$, therefore $(V_1\cup \cdots \cup V_5)\cap V_e=\emptyset$, a contradiction.
$\square$ 

Next we find lower bounds for the total dominator chromatic number of $C_n(1,3)$ when $n\geq 13$.
\begin{lem} \label{TDCN n>=13 except 14,19}
{\rm
For any circulant graph $C_n(1,3)$ of order $n\geq 13$,
\begin{eqnarray*}
\chi_{d}^t(C_n(1,3)\geq \left\{
\begin{array}{ll}
2\lceil \frac{n}{8}\rceil +1     & \mbox{ if }n\equiv 1\pmod{8},\\
2\lceil \frac{n}{8}\rceil +2    & \mbox{ otherwise.}
\end{array}
\right.
\end{eqnarray*}
}
\end{lem}

\textbf{Proof: }
The statement is true for $n=14,18, 19$ by Lemma \ref{TDCN the 7 numbers}. Now assume $n\neq 14,18,19$.
Let $f=(V_1,V_2,\cdots,V_{\ell})$ be a TDC of $C_n(1,3)$, where $v_1\geq v_2\geq \cdots \geq v_{\ell}$ and
\begin{eqnarray*}
\ell = \left\{
\begin{array}{ll}
2\lceil \frac{n}{8}\rceil         & \mbox{ if }n\equiv 1\pmod{8},\\
2\lceil \frac{n}{8}\rceil +1    & \mbox { if }n\not\equiv 1\pmod{8}.
\end{array}
\right.
\end{eqnarray*}

\noindent Let $m=|\{i~|~v_i\geq 5\}|$.
\noindent If $m\geq 3$, we may assume $v_1'=v_2'=v_3'=0$. Then
$
\sum_{i=1}^{\ell}v_i'  =  \sum_{i=4}^{\ell}v_i'
                                           \leq 4(\ell-3)
                                            <  n,
$
a contradiction.

\noindent If $m=0$, then $5\ell< 2n$ and so $\sum\limits_{i=1}^{\ell} v_i'< n$, which is a contradiction with (\ref{sum v_i'>=n}).
Now let $m=1$. Hence, $v_1'=0$.
By Observation \ref{Obs.1}, Proposition \ref{alpha_Cn(1,3)} and Corollary \ref{Max.ind.dif.parity},
\[
\sum\limits_{i=1}^{\ell}v_i'  =  \sum\limits_{i=2}^{\ell}v_i'
                                           \leq  \sum\limits_{i=2}^{\ell}(5-v_i)
                                           =  5(\ell-1)-n+v_1 <  n,
\]
a contradiction with (\ref{sum v_i'>=n}).
Now suppose that $m=2$. If $n\equiv 0,1,5,6,7\pmod{8}$, then
\[
\sum\limits_{i=1}^{\ell}v_i'  =  \sum\limits_{i=3}^{\ell}v_i'
                                           \leq 4(\ell-2)
                                            <  n,
\]
a contradiction with (\ref{sum v_i'>=n}).

The remaining cases are $n\equiv 2,3,4\pmod{8}$ and $m=2$.

\textbf{Case 1.} $n\equiv 3\pmod{8} \mbox { and }m=2$.

\noindent Let $n=8k+3$ for some nonnegative integer $k$ and $m=2$.
Since $m=2$, it follows that $v_1\geq v_2\geq 5$. Note that, by definition, $\ell=2k+3$. Since
$\ell -2\leq n-(v_1+v_2),$ it follows that $v_1+v_2\leq 6k+2$.
If $v_1+v_2\leq 6k$, then
\[
\sum\limits_{i=1}^{\ell}v_i'  =  \sum\limits_{i=3}^{\ell}v_i'
                            \leq \sum\limits_{i=3}^{\ell} (5-v_i)
                            = 5(\ell-2)-n+v_1+v_2 <  n,
\]
which contradicts (\ref{sum v_i'>=n}). We consider two following subcases.
\begin{itemize}
\item{Subcase 1.1.} $v_1+v_2=6k+2$.

\noindent Then $n-(v_1+v_2)=2k+1$ and $V_i=\{w_i\}$ for $3\leq i\leq \ell=2k+3$.
In order to satisfy the inequality $\sum_{i=1}^{\ell}v_i'\geq n$, given in (\ref{sum v_i'>=n}), we must have
$\bigcup_{i=3}^{\ell} N(w_i)= V(C_n(1,3))$.
Hence, there exists a set $L\subset \{w_i: 3\leq i\leq \ell\}$ of cardinality $2k$ which is an open packing in $C_n(1,3)$. Let $\{w_j\}=\{w_i: 3\leq i\leq \ell\}\setminus L$. Obviously, if $i,i+3\in L$ and
$w_j=i+1$ or
$w_j=i+2$ (addition is modulo $n$ with residue in $\{1,2,\ldots,n\}$), then $|N(w_j)\cap N\{i,i+3\}|=3$ by Observation \ref{obs2}. So $\sum_{i=1}^{\ell}v_i'<n$,
a contradiction.
For the other cases we proceed as follows.
Let $a,b,c$ be three distinct vertices of $V(C_n(1,3))$. We say vertex $b$ is between vertices $a$ and $c$ if
the smaller path between $a$ and $c$, on the cycle $(1,2,\ldots, n),$ contains vertex $b$.
Let $S$ be a subset of $L$ with four vertices such that no vertex of $L\setminus S$ is between any two vertices of $S$ and the subgraph induced by $S$ consists of two edges. Note that by Observation \ref{obs3}
the subgraph induced by $L$ consists of $k$ edges. The possibilities for $S$ are
$\{i,i+1,j,j+1\}$, $\{i,i+1,j,j+3\}$ and $\{i,i+3,j,j+3\}$. If $S=\{i,i+1,j,j+1\}$ or $S=\{i,i+1,j,j+3\}$, then, by Observation \ref{obs2}, it is easy to see that for every $i+1 < u < j$, $|N(u)\cap N(L)|\geq 2$. If $S=\{i,i+3,j,j+3\}$, then for every $i+1 < u < j$, $|N(u)\cap N(L)|\geq 2$ provided $j-(i+3)\neq 8$. So in these cases $\sum_{i=1}^{\ell}v_i'<n$, a contradiction. Finally, if $j-(i+3)=8$,
then $|N(i+6)\cap N(L)|=1$ and $|N(i+8)\cap N(L)|=1$. But if $w_j=i+6$ or $w_j=i+8$, then the set
$V\setminus (L\cup\{w_j\})$ cannot be partitioned into two independent sets.

\item{Subcase 1.2.} $v_1+v_2=6k+1$.

\noindent Then $n-(v_1+v_2)=2k+2$, $V_3=\{a,b\}$ and $V_i=\{w_i\}$ for $4\leq i\leq \ell$.
In order to satisfy the inequality $\sum_{i=1}^{\ell}v_i'\geq n$ given in (\ref{sum v_i'>=n}) we must have
$(\bigcup_{i=4}^{\ell} N(w_i))\cup CN(V_3)= V(C_n(1,3))$. Therefore $L=\{w_i\mid 4\leq i\leq \ell\}$
forms an open packing of size $2k$ in $C_n(1,3)$. In addition, $|CN(V_3)|=3$ and
$N(L)\cap CN(V_3)=\emptyset$.
By Observation \ref{obs2}, the equality $|CN(V_3)|=3$ implies that the distance between
vertices $a$ and $b$ on the cycle $(1,2,\ldots,n)$ must be $2$ .
Let $S$ be as in Subcase 1.1. It is easy to see that for every two vertices $a$ and $b$ of distance two
which are between any two vertices of $S$, we obtain
$|N(L)\cap CN\{a, b\}|\geq 1$, which is a contradiction.
\end{itemize}
\textbf{Case 2.} $n\equiv 2, 4\pmod{8} \mbox { and }m=2$.

\noindent Let $n=8k+2$ for some nonnegative integer $k$.
Since $m=2$, it follows that $v_1\geq v_2\geq 5$. Note that, by definition, $\ell=2k+3$. Since
$\ell -2\leq n-(v_1+v_2),$ it follows that $v_1+v_2\leq 6k+1$.
On the other hand, if $v_1+v_2\leq 6k-2$, then
\[
\sum\limits_{i=1}^{\ell}v_i'  =  \sum\limits_{i=3}^{\ell}v_i'
                            \leq \sum\limits_{i=3}^{\ell} (5-v_i)
                            = 5(\ell-2)-n+v_1+v_2 <  n,
\]
which contradicts (\ref{sum v_i'>=n}).

\noindent So we have three subcases to consider, they are, $v_1+v_2=6k+1, 6k$, $6k-1$.
If $v_1+v_2=6k+1$, then $n-(v_1+v_2)=2k+1$.
An argument similar to that described in Subcase 1.1 leads to a contradiction.
If $v_1+v_2=6k$, then $n-(v_1+v_2)=2k+2$.
An argument similar to that described in Subcase 1.2 leads to a contradiction.

Now let $v_1+v_2=6k-1$. Then $n-(v_1+v_2)=2k+3$.
First let $V_3=\{a,b,c\}$.
If $v_3'\leq 1$, then $\sum_{i=1}^{\ell} v_i'\leq 8k+1<n$, a contradiction.
If $v_3'=2$, then $L=\cup_{i=4}^{\ell} V_i$ must be an open packing of cardinality $2k$.
In addition, by Observation \ref{obs2},
the distances between vertices of $V_3$ must be 2, 2 and 4.
So the vertices of $CN(V_3)$ have the same parity. If all are even (the case all are odd is similar), then $CN(V_3)$ has two odd numbers, and hence there will be $4k+2$ odd numbers in $CN(L)\cup CN(V_3)$. This is a contradiction.
Now, without loss of generality, let $|V_3|=|V_4|=2$. Then $L=\bigcup_{i=5}^{\ell} V_i$ must be an open packing of cardinality $2k-1$. Since $n$ is even, the parity of a vertex is different from the parity of its neighbors.
So without loss of generality we may assume $k$ members of $L$ are even and $k-1$ members are odd.
A simple counting argument shows that the members of $V_3$ have the same parity, say even, and the members of
$V_4$ also have the same parity, say odd (the other cases are similar).
Then we will have at most $4k-1$ even numbers in $\bigcup_{i=1}^{\ell} CN(V_i)$, which is a contradiction.

Finally, let $n=8k+4$ for some nonnegative integer $k$.
Since $m=2$, then $v_1\geq v_2\geq 5$. Note that, by definition, $\ell=2k+3$. Since
$\ell -2\leq n-(v_1+v_2),$ it follows that $v_1+v_2\leq 6k+3$.
On the other hand, if $v_1+v_2\leq 6k+2$, then
\[
\sum\limits_{i=1}^{\ell}v_i'  =  \sum\limits_{i=3}^{\ell}v_i'
                            \leq \sum\limits_{i=3}^{\ell} (5-v_i)
                            = 5(\ell-2)-n+v_1+v_2 <  n,
\]
which contradicts (\ref{sum v_i'>=n}).

\noindent Hence, $v_1+v_2=6k+3$. So $n-(v_1+v_2)=2k+1$. In order to satisfy
$\bigcup_{i=3}^{\ell} CN(V_i)=V(C_n(1,3))$, we must have $|V_i|=1$ for $3\leq i\leq \ell$ and
the set $\bigcup_{i=3}^{\ell}V_i$ must be an open packing in $C_n(1,3)$ of cardinality
$2k+1$, which is impossible by Proposition \ref{open-packing Cn(1,3)}.
$\square$ 


\section{Proof of Theorem \ref{chi_{d}^t(C_n(1,3))}}

\noindent First note that $C_6(1,3)$ is isomorphic to the complete bipartite graph $K_{3,3}$ and so $\chi_d^t(C_6(1,3))$ $=\chi(K_{3,3})=2$. Now for each $n\geq 7$ we present color classes for $C_n(1,3)$ which satisfy the equality in Lemma \ref{TDCN the 7 numbers} when $7\leq n\leq 12$, $n=14$ or $n=19$ and the equality in
Lemma \ref{TDCN n>=13 except 14,19} when $n\geq 13$ and $n\neq 14,19$. Hence, the proof of the theorem follows.

\noindent If $7\leq n\leq 11$, then the color classes are:\\
\begin{tabular}{ll}
$n=7:$ & $\{1\}, \{2,7\}, \{3,5\}, \{4,6\},$\\
$n=8:$ & $\{1,3,5,7\}, \{2,4,6,8\},$\\
$n=9:$&   $\{1,8\}, \{2,9\}, \{3,5,7\}, \{4,6\},$\\
$n=10:$ & $\{1\}, \{2\}, \{3,5,7,9\}, \{4,6,8,10\},$\\
$n=11:$& $\{1,3,5\}, \{2,11\}, \{7,9\}, \{8,10\}, \{4,6\}.$
\end{tabular}

\noindent If $n\geq 12$, we proceed as follows. Let $E$ and $O$ consist of even and odd numbers less
than or equal to $n$, respectively.
Define
$$
\begin{array}{lll}
A_1=\{n-6,n-4,n-2,n\},&&
A_2 = \{n-7, n-5, n-3, n-1\},\\
A_3 = \{n-6, n-4, n-2\},&&
A_4 = \{n-7, n-5, n-3\},\\
A_5 = \{n-6, n-4\}, &&
A_6 = \{n-7, n-5\}.
\end{array}
$$
\noindent Consider the open packing $L=\{8i+1,8i+2\mid 0\leq i\leq k-1\}$. The color classes are:
$n=8k:$\hspace{7mm} $\{v\} \mbox{ for } {v\in L}, O\setminus L, E\setminus L,$\\
$n=8k+1:$ $\{v\} \mbox{ for } {v\in L},A_1,O\setminus(L\cup A_1), E\setminus L,$\\
$n=8k+2:$ $\{v\} \mbox{ for } {v\in L},A_1,A_2,O\setminus(L\cup A_2), E\setminus (L\cup A_1),$\\
$n=8k+3:$ $\{v\} \mbox{ for } {v\in L},A_2,A_3, O\setminus(L\cup A_3), (E\setminus (L\cup A_2))\cup\{n\},$\\
$n=8k+4:$ $\{v\} \mbox{ for } {v\in L},A_3,A_4, O\setminus(L\cup A_4), E\setminus (L\cup A_3),$\\
$n=8k+5:$ $\{v\} \mbox{ for } {v\in L},A_4,A_5, (O\setminus(L\cup A_5))\cup\{n-1\}, (E\setminus (L\cup A_4))\cup\{n-2,n\},$\\
$n=8k+6:$\ $\{v\} \mbox{ for } {v\in L},A_5,A_6, O\setminus(L\cup A_6), E\setminus (L\cup A_5),$\\
$n=8k+7:$ $\{v\} \mbox{ for } {v\in L},\{n-6\}, A_6, (O\setminus L\cup \{n-6\})\cup\{n-3, n-1\}, (E\setminus (L\cup A_6))\cup\{n-4, n-2, n\}.$

Theorems \ref{chi_{d}^t(C_n(1,3))} and \ref{TDN C_n(1,3)}
lead to the following corollary.

\begin{cor}\label{corollary 2}
{\rm
\begin{eqnarray*}
\chi_{d}^t(C_n(1,3))=\left\{
\begin{array}{ll}
\gamma_t(C_n(1,3))      & \mbox{ if }n=8,10 \\
\gamma_t(C_n(1,3)) +1 & \mbox{ if }n=9 \\
\gamma_t(C_n(1,3)) +3 & \mbox{ if }n\equiv 3\pmod{8}, n\neq 11\\
\gamma_t(C_n(1,3)) +2 & \mbox{ otherwise.}
\end{array}
\right.
\end{eqnarray*}
}
\end{cor}


\end{document}